\documentclass[12pt]{amsart}
\usepackage{amssymb}
\usepackage{epsfig}
\usepackage{graphicx}

\theoremstyle{plain}
\newtheorem{prop}{Proposition}
\newtheorem{theo}[prop]{Theorem}
\newtheorem{coro}[prop]{Corollary}

\theoremstyle{definition}
\newtheorem{defi}[prop]{Definition}

\newtheorem{conj}[prop]{Conjecture}

\newtheorem{rema}[prop]{Remark}

\newtheorem{exam}[prop]{Example}

\newcommand{\Sym}{\mathrm{Sym}}
\newcommand{\Pic}{\mathrm{Pic}}

\newcommand{\NE}{\mathrm{NE}}
\newcommand{\rN}{\mathrm{N}}
\newcommand{\oNE}{\overline{\mathrm{NE}}}

\newcommand{\cBK}{\mathcal{BK}}
\newcommand{\ocBK}{\overline{\mathcal{BK}}}
\newcommand{\ocK}{\overline{\mathcal{K}}}
\newcommand{\ocC}{\overline{\mathcal{C}}}

\newcommand{\nod}{\mathrm{nod}}
\newcommand{\HKT}{\mathrm{K3}^{[2]}}

\newcommand{\ra}{\rightarrow}

\newcommand{\bP}{{\mathbb P}}

\newcommand{\bC}{{\mathbb C}}

\newcommand{\bG}{{\mathbb G}}

\newcommand{\bN}{{\mathbb N}}
\newcommand{\bQ}{{\mathbb Q}}
\newcommand{\bR}{{\mathbb R}}

\newcommand{\bZ}{{\mathbb Z}}
\newcommand{\cC}{{\mathcal C}}

\newcommand{\cF}{{\mathcal F}}

\newcommand{\cK}{{\mathcal K}}

\newcommand{\cN}{{\mathcal N}}
\newcommand{\cO}{{\mathcal O}}

\newcommand{\cT}{{\mathcal T}}

\author{Brendan Hassett and Yuri Tschinkel}
\title[Moving and ample cones]{Moving and ample cones of holomorphic symplectic fourfolds}

\begin{document}
\begin{abstract}
We analyze the ample and moving cones of holomorphic symplectic manifolds, in light
of recent advances in the minimal model program.  As an application, we establish a numerical 
criterion for ampleness of divisors on fourfolds deformation-equivalent to punctual Hilbert schemes of 
K3 surfaces.  
\end{abstract}
\date{\today}

\maketitle

\section{Introduction}
Let $F$ be an irreducible holomorphic symplectic variety.
The integral cohomology
group $H^2(F,\bZ)$ is equipped with a natural integral
quadratic form $\left(,\right)$, known as the Beauville-Bogomolov
form; the complex structure on $F$ induces a Hodge structure
on the complex cohomology group $H^2(F,\bC)$.  
The birational geometry of $F$ is tightly coupled to these 
cohomological invariants. 
For example, the ample cone of a polarized K3 surface $(F,g)$
is explicitly determined by these structures on $H^2(F,\bZ)$
and the class $[g]$ \cite[\S 2]{LP}.

In higher-dimensions, qualitative descriptions of the ample
cone have been obtained by Huybrechts \cite{Hu03}
and Boucksom \cite{Boucksom}.  However, these fall short of
the precise picture we have for K3 surfaces.  
The birational geometry of higher-dimensional varieties is
much richer than surfaces, and this is reflected in the
numerous invariants we can assign to these, e.g.,
the {\em moving cone} parametrizing effective divisors
without fixed components.  Again, we have a qualitative
description of this due to Huybrechts \cite{Hu03} (and
Theorem~\ref{theo:SIMD} below), but this is not sufficient
to determine whether a given divisor is moving.  

The last ten years have seen great advances in the birational
geometry of holomorphic symplectic varieties and the classification
of their birational contractions, especially in
the four-dimensional case \cite{WW} \cite{W} \cite{Nami} 
\cite{BHL} \cite{CMSB}.  This is therefore
a natural testing-ground for conjectures on the shape of the ample
and moving cones.  For fourfolds deformation-equivalent
to the punctual Hilbert scheme of a K3 surface, we formulated in 1999
precise conjectures characterizing the ample and moving cones \cite{HT01}.  

The recent progress in the log minimal model program
\cite{BCHM} lends new impetus to the efforts to resolve these problems.
Our principal result is Theorem~\ref{theo:main}, which gives one implication
of our conjecture, namely, that
the divisors claimed to be ample are in fact ample.
This entails a numerical classification of extremal rays, given in 
Theorem~\ref{theo:classifyray}.  We also prove general results about the
structure of ample and moving cones on irreducible holomorphic 
symplectic varieties 
of arbitrary dimension.

The paper is organized as follows: in Section~\ref{sect:generalities} 
we recall basic notation, constructions and conjectures 
relating to holomorphic symplectic fourfolds.  Section~\ref{sect:LogMMP}
outlines applications of the minimial model program to our situation.
In Section~\ref{sect:conj} we specialize to the four-dimensional
example mentioned above and recall the conjectures of \cite{HT01}.  
Section~\ref{sect:scientific} offers an analysis `from first
principles' of cohomology classes of extremal rays.   Finally, 
Section~\ref{sect:half} contains the proof of the main theorem.  

\

\noindent {\bf Acknowledgments:}  We are grateful to Daniel
Huybrechts and Dmitry Kaledin for valuable conversations.
The first author was partially supported by National Science Foundation Grants
0554491 and 0134259.
He appreciates the hospitality of the University of G\"ottingen.  
The second author was partially supported by National Science Foundation
Grants 0554280 and 0602333.

\section{Generalities on ample cones of holomorphic symplectic manifolds}
\label{sect:generalities}

Let $F$ be a projective irreducible holomorphic variety,
$\rN_1(F,\bZ)$ the group of one-cycles (up to numerical
equivalence), $\rN^1(F,\bZ)$ the group of divisor-classes,
$\NE_1(F)\subset \rN_1(F,\bR)$
the cone of effective curves, and $\oNE_1(F)$ its
closure.  The dual to $\oNE_1(F)$ in $\rN^1(F,\bR)$ is the {\em
nef cone} of $F$.  
Recall that
$\bR_{\ge 0} \varrho \subset \oNE_1(F)$ is an {\em extremal ray}
if whenever $\varrho=c_1C_1+c_2C_2$ for
$C_1,C_2 \in \oNE_1(F)$ and $c_1,c_2 > 0$ then 
$C_1,C_2 \in \bR_{\ge 0}\varrho$.  

Let $\left(,\right)$ denote the Beauville-Bogomolov form on
$H^2(F,\bZ)$, normalized so that it is integral but not divisible.
The induced $\bQ$-valued form on $H_2(F,\bZ)$ is also denoted
$\left(,\right)$.  Let $\cC_F$
denote the connected
component of the positive cone of $F$ 
$$\{ \alpha \in H^2(F,\bR) \cap H^{1,1}(F,\bC): \left(\alpha,\alpha\right)
>0 \}$$
containing the polarization.  

We recall some general facts:
\begin{itemize}
\item{The form $\left(,\right)$ has signature $(3,\dim H^2(F,\bR)-3)$ on $H^2(F,\bR)$
and signature $(1,\dim H^2(F,\bR)-3)$ on $H^2(F,\bR) \cap H^{1,1}(F,\bC)$ \cite[1.9]{Hu99}}.
\item{Fujiki \cite{Fujiki} \cite[1.11]{Hu99}
showed there exists a positive constant $c_0$ 
such that for each $\alpha \in H^2(F,\bR)$ 
$$\alpha^{\dim(F)}=c_0 \left(\alpha,\alpha\right)^{\dim(F)/2}.$$
More generally, for each Chern class $c_i(F)$ there exists
a constant $c_i$ such that
$$c_i(F)\alpha^{\dim(F)-i}=c_i \left(\alpha,\alpha\right)^{(\dim(F)-i)/2}.$$
}
\item{
Each divisor class $D$ with $\left(D,D\right)>0$ is big
\cite[3.10]{Hu99} \cite{Hu03b}}.
\item{There is an 
integral formula for the Beauville-Bogomolov form
\cite[\S 1.9]{Hu99} \cite{Bea85}. 
Choose $\sigma \neq 0 \in \Gamma(F,\Omega^2_F)$, normalized such that 
$$\int_F (\sigma \bar{\sigma})^{\dim(F)}=1.$$
Then there exists a positive real constant $c$ such that
\begin{equation} \label{eqn:integral}
\left(\alpha,\beta\right)=
		c \int_F \alpha \beta (\sigma \bar{\sigma})^{\dim(F)-1}
\end{equation}
for all $\alpha,\beta \in H^{1,1}(F,\bC)$.  }
\item{Let $D$ be a nef and big divisor class.
By Kawamata-Viehweg vanishing \cite[\S 8]{CKM}, $D$
has no higher cohomology.
By basepoint-freeness \cite[\S 9]{CKM}, $ND$ is globally
generated for some $N\gg 0$.}
\end{itemize}

Let 
$$\cK_F \subset \cC_F, \quad  \ocK_F\subset \ocC_F$$
denote the {\em
K\"ahler cone} of $F$ and its closure.  
The intersection $\cK_F \cap H^2(F,\bZ)$ (resp.
$\ocK_F \cap H^2(F,\bZ)$) is the set of ample (resp. nef)
divisors on $F$.

We recall the following result \cite[\S 5]{Hu99}:
\begin{theo} \label{theo:makeKahler}
Choose $\alpha \in \cC_F$ `very general', e.g., not orthogonal
to any integral class, cf. \cite[5.9]{Hu99}.  Then there
exists an irreducible holomorphic variety $F'$ and correspondence 
$\Gamma \subset F \times F'$ inducing a 
birational map $\phi:F' \dashrightarrow F$ such that 
\begin{itemize}
\item{$\Gamma_*:H^2(F,\bZ) \ra H^2(F',\bZ)$ is an isomorphism
respecting the Beauville-Bogomolov forms;}
\item{$\Gamma_* \alpha \in \cK_{F'}$.}
\end{itemize}
\end{theo}
The correspondence $\Gamma$ is the specialization of the
graph of an isomorphism $F'_t \stackrel{\sim}{\ra} F_t$,
where $F'_t$ and $F_t$ are fibers of small deformations 
$$\cF',\cF \ra \{t:|t|<1 \}$$
of $F'$ and $F$ respectively.  

\begin{exam}
The simplest nontrivial example is the Atiyah flop.
Let $F$ be a K3 surface
containing a $(-2)$-curve $E$ and 
$$\cF \ra \{t:|t|<1 \}$$
a general deformation of $F$, so the class $[E]$ does not
remain algebraic.  Let
$$\cF' \ra \{t:|t|<1 \}$$
denote the flop of $E$;  the fiber $F'$ over $t=0$ contains
a $(-2)$-curve $E'$.  
Note that in this case
$\phi:F'\stackrel{\sim}{\ra} F'$ but
$$\Gamma=\mathrm{Graph}(\phi)+E \times E' \subset F\times F'.$$
\end{exam}

\begin{rema}
From our example, it is evident that 
$$\phi^*\alpha\neq \Gamma_*\alpha$$
in general.  Equality holds iff
$$\Gamma=\mathrm{Graph}(\phi)+\sum_i Z_i$$
where each $Z_i$ maps to a codimension $\ge 2$ subvariety
in each factor.  
\end{rema}

Let
$$\ocBK_F \subset \ocC_F \subset H^2(F,\bR)\cap H^{1,1}(F,\bC)$$
denote the closure of the {\em birational 
K\"ahler cone} of $F$, i.e., the union $\cBK_F$ of the K\"ahler cones of
all holomorphic symplectic varieties birational to $F$.
This has the following numerical interpretation:
\begin{prop} \cite[4.2]{Hu03}
A class $\alpha \in \ocC_F$
lies in $\ocBK_F$ if and only if $\left(\alpha,D\right)\ge 0$
for each uniruled divisor $D \subset F$.  
\end{prop}
\begin{defi} An effective divisor $M$ on $F$ is {\em moving}
if there exists an integer $N>0$ such that $NM$ has no fixed components.
The {\em moving cone} of $F$ is the cone generated by
moving divisors.
\end{defi}
\begin{rema}\label{rema:BKinMOV} It is clear from the definition that 
elements of the birational K\"ahler cone $\cBK_F$ are contained in the 
moving cone.
\end{rema}

\begin{theo}[Symplectic interpretation of moving divisors] \label{theo:SIMD}
Each moving divisor is contained in the closure of
the birational K\"ahler cone $\ocBK_F$.  
Thus $\ocBK_F$ equals the closure of the moving cone of $F$.
\end{theo}
\begin{rema}
Corollary~\ref{coro:converse} below is a partial converse to this result.
\end{rema}
\begin{proof}
We are grateful to Professor D. Huybrechts for his help with this
argument.  

Suppose that $M$ is moving.  To show that $M \in \ocBK_F$,
it suffices that $\left(M,D\right)\ge 0$ for each irreducible uniruled divisor
$D\subset F$.  We write $n=\dim(F)$.  

Replacing $M$ by a suitable multiple if necessary,
we may assume that $M$ has no fixed components,
i.e., its base locus has codimension
at least two.  There exists a diagram 
$$\begin{array}{ccc}
Z & \stackrel{p}{\ra} & F' \\
{\scriptstyle q}\downarrow \ & & \\
F & & 
\end{array}
$$
where $Z\ra F$ is a smooth projective resolution of the base locus
of $|M|$ and $p$ is the resulting morphism.  Thus there exists
an ample line divisor $H$ on $F'$ such that
$$q^*M=\sum_i c_i E_i + p^*H$$
where each $c_i\ge 0$ and $E_i$ is a $q$-exceptional divisor in $Z$.

Compute the Beauville-Bogomolov form by pulling back to $Z$:
$$\begin{array}{rcl}
\left(M,D\right) &=&c\int_F [M] [D] (\sigma \bar{\sigma})^{n-1} \\
	     	&=&c \int_Z q^*[M] q^*[D] q^*((\sigma \bar{\sigma})^{n-1})\\
		&=& c \int_Z (\sum_i c_i [E_i] + p^*[H])(q^*[D])
				q^*((\sigma \bar{\sigma})^{n-1}).
\end{array}
$$
First, note that
$$\int_Z [E_i] q^*[D] q^*((\sigma \bar{\sigma})^{n-1})=0.$$
Indeed, any degree-$(4n-2)$ form pulled back from $F$ integrates
to zero along $E_i$ because 
$\mathrm{codim}_{\bR}q(E_i)\ge 2$.  To evaluate the second term
$$\int_Z p^*[H] q^*[D] q^*((\sigma \bar{\sigma})^{n-1}),$$  
observe that the intersection $p^*[H] \cap q^*[D]$ involves a 
semiample divisor and an effective divisor.  In particular,
we can express
$$p^*[H] \cap q^*[D]=\sum_j n_j W_j, \quad n_j>0,$$
where each $W_j$ is a $(2n-2)$-dimensional subvariety of $Z$.  
Thus we have
$$\int_Z p^*[H] q^*[D] q^*((\sigma \bar{\sigma})^{n-1})=
\sum_j n_j \int_{W_j}q^*((\sigma \bar{\sigma}))^{n-1}.$$
Let $\tilde{W}_j\ra W_j$ denote a resolution of singularities 
and $r:\tilde{W}_j \ra F$ the induced morphism.  
We have
$$\int_{W_j}q^*((\sigma \bar{\sigma})^{n-1})=
\int_{\tilde{W}_j} (r^*\sigma \overline{r^*\sigma})^{n-1} \ge 0$$
because the integrand is a nonnegative multiple of the volume 
form on $\tilde{W}_j$. 
\end{proof}

We have the following result of Boucksom \cite{Boucksom} and Huybrechts \cite[\S 3]{Hu03}:
\begin{theo} \label{theo:Boucksom}
A class $\alpha \in \cC_F$ (resp. $\ocC_F$) 
is in $\cK_F$
(resp. $\ocK_F$) 
if and only if $\alpha.C>0$ (resp. $\alpha.C \ge 0$)
for each rational
curve $C\subset F$.
\end{theo}
However, this does not imply, {\em a priori}, that these classes 
determine a locally-finite rational polyhedral cone, nor does it 
provide a geometric interpretation of these rational curves.

The Cone Theorem does shed some light on this. 
\begin{prop}[Cone Theorem for varieties with trivial canonical class] \cite[3.7]{KM}
Let $Y$ be a smooth projective variety with $K_Y=0$
and $\Delta$ an effective $\bQ$-divisor on $Y$ such that
$(Y,\Delta)$ has Kawamata log terminal singularities
(see \cite[2.13]{FA} for the definition.)  
Then the closed cone of effective curves $\oNE_1(Y)$ can be expressed
$$\oNE_1(Y)=\oNE_1(Y)_{\Delta.C\ge 0}+ \sum_j \bR_{\ge 0}[C_j], \quad
\Delta.C_j < 0$$
where the $C_j$ are extremal and represent rational curves collapsed 
by contractions of $Y$.  This is locally
finite in the following sense:
For an ample divisor $A$ and $\epsilon>0$, there are a finite number 
of $C_j$ with $C_j.(\Delta+\epsilon A)<0$.  
\end{prop}

Which parts of $\oNE_1(F)$ can be analyzed using this fact?

\begin{prop} \label{prop:polyhedralifmoving}
Let $M$ be a divisor class contained in $\ocK_F$ and in the interior of $\ocBK_F$.  
Then $\ocK_F$ is locally-finite rational polyhedral in a neighborhood
of $M$.
\end{prop}
\begin{proof} There is nothing to prove unless $M$ lies on the boundary
of $\ocK_F$.

We translate the statement using the duality between curves and divisors:
Suppose that $R$ is an extremal ray of $\oNE_1(F)$ such that
$R^{\perp}$ meets the interior of $\ocBK_F$.  We claim that
$\oNE_1(F)$ is finite rational polyhedral in a neighborhood
of $R$.

Each $M$ in the interior of $\ocBK_F$ is contained in $\cC_F$, i.e.,
$\left(M,M\right) > 0$.  
Since $R$ is orthogonal to $M$ we have
$\left(R,R\right)<0$.  

By Theorem~\ref{theo:SIMD}, there exists 
an $\alpha \in \cBK_F$ with $\alpha R <0$.  We may assume
$\alpha$ is `general' in the sense of 
Theorem~\ref{theo:makeKahler}.  Let
$F'$ be the hyperk\"ahler manifold birational to $F$
with K\"ahler class $\alpha$.
Suppose that $A'$ is a
very ample divisor of $F'$ and let
$A$ be its proper transform in $F$.  We have $A R<0$ as well.  

Choose $\epsilon >0 \in \bQ$
such that $\epsilon A$ is Kawamata log terminal.  Indeed,
since $F$ is smooth if we choose
$\epsilon$ such that
$$1/\epsilon > \max_{x \in F} \{ \mathrm{mult}_x(A) \}$$
then the singularities are Kawamata log terminal by \cite[8.10]{KSoP}. 
Then the Cone Theorem implies that the effective cone
is finite polyhedral in some neighborhood of $R$.
\end{proof}

\section{Application of the log minimal model program}
\label{sect:LogMMP}

We will use the following consequence of the log minimal
model program:
\begin{theo} \label{theo:fingen}
Let $Y$ be a smooth projective variety with $K_Y$ trivial.
Suppose that $D_1,\ldots,D_r$ are big divisors on $Y$.
Then the ring
$$\oplus_{(n_1,\ldots,n_r)\in \bZ_{\ge 0}^r} 
\Gamma(F,\cO_F(n_1D_1+\ldots+n_rD_r))$$
is finitely generated.
\end{theo}
\begin{proof}
There exists a positive $\epsilon \in \bQ$ such that each 
$\epsilon D_i$ has divisorial log terminal singularities
(see \cite[2.13]{FA} for the definition.)   Indeed,
if we choose $\epsilon$ such that
$$1/\epsilon > \max_{y \in Y,i=1,\ldots,r} \{ \mathrm{mult}_y(D_i,y) \}$$
then \cite[8.10]{KSoP} guarantees the singularities have the desired property.
It follows from \cite[1.1.9]{BCHM} that the graded ring
$$\oplus_{(m_1,\ldots,m_r) \in \bZ_{\ge 0}^r}
\Gamma(F,\cO_F(\lfloor \sum_i m_i \epsilon D_i\rfloor))
$$
is finitely generated.  It remains finitely generated
when we restrict to the multidegrees such that
each $m_i\epsilon \in \bZ$.  
\end{proof}

\begin{prop} \label{prop:chamber}
Let $F$ be a projective irreducible holomorphic symplectic manifold.  
Consider the chamber decomposition
$$\cup_{F'} \cK_{F'} \subset \ocBK_F$$
where the union is taken over holomorphic symplectic birational models of $F$.
This is locally finite polyhedral near divisors $M \in \cC_F$, 
i.e., given a small
neighborhood $U \ni M$, the cone $\ocBK_F \cap U$ is defined in $U$ by 
a finite number of rational linear inequalities, with
chambers equal to the complement to a finite number of rational
hyperplanes. 
\end{prop}
In particular, $\ocK_F \cap \cC_F$ is locally-finite 
rational polyhedral at each divisor
$M \in  \ocK_F \cap \cC_F$.  
\begin{proof}
Recall that divisors $M$ with $\left(M,M\right)>0$ are big.  Since being 
big is an open condition in $\rN^1(F,\bR)$, 
we can express $M$ as an element of the convex hull of a collection of
big divisors $D_1,\ldots,D_r$ which freely generate $\rN^1(F,\bZ)$.  
Consider the intersection of $\ocBK_F$ with the cone
$$\left<D_1,\ldots,D_r\right>;$$
it suffices to verify that this is locally finite polyhedral.

The graded ring associated to these divisors 
$$R(D_1,\ldots,D_r):=\oplus_{(n_1,\ldots,n_r)\in \bZ_{\ge 0}^r} 
\Gamma(F,\cO_F(n_1D_1+\ldots+n_rD_r))$$
is finitely generated by
Theorem~\ref{theo:fingen}.
As discussed in \cite[2.9]{HuKeel}, this finite generation has 
implications for the birational geometry of $F$:  
\begin{itemize} 
\item{the subcone
$$\ocK_F \cap \left<D_1,\ldots,D_r\right> \subset \left<D_1,\ldots,D_r\right>$$
is determined by a finite number of linear rational inequalities;}
\item{the moving divisors in $\left<D_1,\ldots,D_r\right>$ are the union
of the rational polyhedral cones
$$\cup_{F'} \ocK_{F'} \cap \left<D_1,\ldots,D_r \right>$$
where each $F \stackrel{\sim}{\dashrightarrow}F'$ is a small birational modification.}
\end{itemize}
Indeed, the chamber decompositions of $\left<D_1,\ldots,D_r\right>$
are governed by the various Geometric Invariant Theory quotients of 
$R(D_1,\ldots,D_r)$ under the $\bG_m^r$-action associated
with the multigrading.  We consider linearizations of the action 
corresponding to positive characters of $\bG_m^r$.

Theorem~\ref{theo:SIMD} implies each 
$F'$ is also a projective irreducible holomorphic symplectic manifold,
which completes the proof.
\end{proof}

\begin{coro} \label{coro:curvepolyhedral}
Let $F$ be a projective irreducible holomorphic symplectic manifold.  Then the intersection
$$\oNE_1(F) \cap \{R \in H_2(F,\bR): \left(R,R\right)<0 \}$$
is locally-finite rational polyhedral.  
\end{coro} 
\begin{proof}
As in the proof of Proposition~\ref{prop:polyhedralifmoving}, supporting 
hyperplanes to $\oNE_1(F)$ in the region 
$$\{R: \left(R,R\right) < 0 \}$$
correspond to divisor classes $M$ with $\left(M,M\right)>0$, and Proposition~\ref{prop:chamber}
applies.  
\end{proof}

\begin{coro} \label{coro:converse}
Let $F$ be a projective irreducible holomorphic symplectic manifold.  
Each divisor $M \in \ocBK_F \cap \cC_F$ is moving.  
\end{coro}
\begin{proof}
Proposition~\ref{prop:chamber} implies that $M$ corresponds to a nef
and big divisor $M'$ on some small birational modification
$F\stackrel{\sim}{\dashrightarrow} F'$, where 
$F'$ is a projective irreducible holomorphic symplectic manifold.
Thus basepoint-freeness implies that some multiple of $M'$ is basepoint
free.  Since $F$ and $F'$ are isomorphic in codimension one, we conclude
that $M$ is moving on $F$.
\end{proof}

\begin{rema}
This analysis only applies to divisor classes with {\em positive} 
Beauville-Bogomolov form.  The case where the form is zero is remains
open (cf. Conjecture~\ref{conj:abelian}).
\end{rema}

\section{Conjectures for four-dimensional manifolds}
\label{sect:conj}
In this section, we recall a conjecture on the ample cones
of polarized varieties $(F,g)$ deformation equivalent to $S^{[2]}$,
the Hilbert scheme of length-two subschemes on a K3 surface
$S$.  This was first formulated in \cite{HT01}.  

Denote by
$$
\rN^1_+(F,g) = \{v\in \rN^1(F,\bZ)\,|\, \left(v,g\right) >0\}
$$
the positive halfspace (with respect to $g$ and the Beauville-Bogomolov form).
Let $E$ be the set of classes
$\rho \in \rN^1_+(F,g)$
satisfying one of the following:
\begin{enumerate}
\item $\left(\rho,\rho  \right)=-2$ and $\left(\rho,H^2(F,\bZ)\right)=2\bZ$,
\item $\left(\rho,\rho  \right)=-2$ and $\left(\rho, H^2(F,\bZ)\right)=\bZ$,
\item $\left(\rho,\rho  \right)=-10$ and $\left(\rho,H^2(F,\bZ)\right)=2\bZ$,
\end{enumerate}
Let $E^*$ be the corresponding classes $R\in H_2(F,{\bZ})$, i.e.,
for some $\rho \in E$ we have
$$\left( v,\rho \right)=\begin{cases}
                        R.v & \text{ where } \left(\rho,H^2(F,\bZ)\right)={\bZ} \\
                       2R.v & \text{ where } \left(\rho,H^2(F,\bZ) \right)=2{\bZ}
                       \end{cases}
$$
for each $v\in H^2(F,\bZ)$.  In particular, $R$ satisfies one of the following
\begin{enumerate}
\item $\left(R,R  \right)=-\frac{1}{2}$,
\item $\left(R,R  \right)=-2$,
\item $\left(R,R  \right)=-\frac{5}{2}$.
\end{enumerate}
Let $\rN_E(F,g)\subset H_2(F,\bR)$ be the
smallest real cone containing $E^*$ and the elements  $R\in \rN_1(F,\bZ)$
such that $R.g>0$ and the corresponding $\rho$ has nonnegative
square.

\begin{conj}[Effective curves conjecture]
\label{conj:main}
$$\NE_1(F)=\rN_E(F,g).$$
\end{conj}

The classes in $E^*$ that are extremal in (the closure of) $\rN_E(F,g)$
will be called {\em nodal classes} (cf. \cite[1.4]{LP}).
The nodal classes are denoted $E^*_{\nod}$ and the
corresponding classes in $E$ are denoted $E_{\nod}$.

\begin{conj}[Nodal classes conjecture]\label{conj:nodal}
Each nodal class $R\in E^*_{\nod}$ represents a rational curve contracted
by a birational morphism $\beta$ given by sections of 
$\cO_F(m\lambda),m\gg 0$,
where $\lambda$ is any nef and big divisor class with $R.\lambda =0$.  
\begin{enumerate}
\item If $\left(R,R\right)=-\frac{1}{2},-2$
(i.e., the corresponding $\rho$ is a $(-2)$-class) then $\rho$
is represented by a family of rational curves parametrized
by a K3 surface, which blow down to 
rational double points.
\item If $\left(R,R\right)=-\frac{5}{2}$
(i.e., the corresponding $\rho$ is a $(-10)$-class) then $\rho$
is represented by a family of lines
contained in a $\bP^2$
contracted to a point.
\end{enumerate}
\end{conj}

The remaining generators of the cone of curves are given by:
\begin{conj}\label{conj:abelian} \cite[3.8]{HT01}
Let $\lambda$ be a primitive class on the boundary of
the nef cone with $\left(\lambda,\lambda\right)=0$.  
Then the corresponding line bundle
$\cO_F(\lambda)$ has no higher cohomology and its
sections yield a morphism
$$F \ra {\bP}^2$$
whose generic fiber is an abelian surface.
\end{conj}
This was subsequently generalized to higher dimensions by
Huybrechts \cite{GHJ}.

\section{Deriving $(-2)$ and $(-10)$-classes from first principles}
\label{sect:scientific}
In this section, we give a conceptual explanation for the occurence
of $(-2)$ and $(-10)$-classes in our conjectural description
of the ample cone.  This description is crucial for the proof of our
main theorem in Section~\ref{sect:half}.  

We recall a definition due to O'Grady \cite{ogrady}.  All products of
cohomology classes are to be taken in the cohomology ring unless otherwise
specified:
\begin{defi}
A {\em numerical $\HKT$} is an irreducible holomorphic symplectic fourfold $F$ such that
there exists an isomorphism
$$\psi:H^2(F,\bZ) \stackrel{\sim}{\ra} H^2(S^{[2]},\bZ)$$
with $\psi(\alpha)^4=\alpha^4$ for each $\alpha \in H^2(F,\bZ)$.
Here $S$ is a K3 surface and $S^{[2]}$ its Hilbert scheme of length-two subschemes.
\end{defi}
We recall one key
property  of these manifolds proved by O'Grady \cite[\S 2]{ogrady}:
$H^2(F,\bZ)$ admits a canonical integral primitive quadratic form $\left(,\right)$, the
{\em Beauville-Bogomolov form}, such that
\begin{equation}
H^2(F,\bZ)_{\left(\, , \, \right)}\simeq  U^{\oplus 3} \oplus_{\perp} (-E_8)^{\oplus 2}
\oplus_{\perp} (-2) \label{eqn:thelattice}
\end{equation}
where $U$ is the hyperbolic plane and $E_8$ the positive-definite integral lattice
associated to the corresponding root system.  Moreover, we have
$$\alpha^4=3\left(\alpha,\alpha\right)^2$$
for each $\alpha \in H^2(F,\bZ).$
This form a induces $\frac{1}{2}\bZ$-valued quadratic form on
$H_2(F,\bZ)$ by duality:
$$H_2(F,\bZ)_{\left(\, , \,\right)}\simeq  U^{\oplus 3} \oplus_{\perp} (-E_8)^{\oplus 2}
\oplus_{\perp} (-1/2).$$

We recall additional properties of numerical $\HKT$'s due to O'Grady
\cite[\S 2]{ogrady}:
\begin{itemize}
\item{The intersection product induces an isomorphism 
$$
\Sym^2 H^2(F,\bQ) \stackrel{\sim}{\ra} H^4(F,\bQ) 
$$
and the intersection form on the middle cohomology is given by the formula
$$\alpha_1\alpha_2.\alpha_3\alpha_4=
\left(\alpha_1,\alpha_2\right)\left(\alpha_3,\alpha_4\right)+
\left(\alpha_1,\alpha_3\right)\left(\alpha_2,\alpha_4\right)+
\left(\alpha_1,\alpha_4\right)\left(\alpha_2,\alpha_3\right)$$
for all $\alpha_1,\alpha_2,\alpha_3,\alpha_4 \in H^2(F,\bZ).$}
\item{There is a distinguished class $q^{\vee} \in H^4(F,\bQ) \cap H^{2,2}(F,\bC)$
such that
$$q^{\vee}.\alpha_1.\alpha_2= 25\left(\alpha_1,\alpha_2\right)$$
for all $\alpha_1,\alpha_2 \in H^2(F,\bZ)$.  This is a rational multiple of the dual
Beauville-Bogomolov form induced on $H_2(F,\bZ)$ via Poincar\'e duality.
The class $q^{\vee}$ is the unique Hodge class (up to scalar multiplication)
in the middle cohomology of a general numerical $\HKT$.}
\item{We have the formulas
$$c_2(F)=\frac{6}{5} q^{\vee}, \quad q^{\vee}.q^{\vee}=23 \cdot 25.$$}
\end{itemize}

\begin{theo} \label{theo:classifyray}
Let $F$ be a numerical $\HKT$.  Suppose
$R\in \rN_1(F,\bZ)$ is an extremal ray such that
there exists a Kawamata log terminal effective divisor $B\subset F$
with $B.R < 0$.
If $\left(R,R\right)<0$
then we have
$$\left(R,R\right)=-1/2,-2,-5/2.$$
Moreover, $\rN^1(F,\bZ)$ contains an element $\rho$ satisfying
one of the following:
\begin{itemize}
\item{$\left(\rho,\rho\right)=-2$ and $\left(\rho,H^2(F,\bZ)\right)=\bZ$;}
\item{$\left(\rho,\rho\right)=-2$ and $\left(\rho,H^2(F,\bZ)\right)=2\bZ$;}
\item{$\left(\rho,\rho\right)=-10$ and $\left(\rho,H^2(F,\bZ)\right)=2\bZ$.}
\end{itemize}
\end{theo}
\begin{proof}
The Cone Theorem \cite[3.7]{KM} implies that there exists an extremal
contraction 
$\beta:F \ra F'$ with $\beta_*R=0$ and $\Pic(F/F')\simeq \bZ.$
Since $\left(R,R\right)<0$ the line bundle $L$ contracting
$\rho$ has $\left(L,L\right)>0$;  indeed, this is a consequence
of the fact that $\left(,\right)$ has signature 
$(1,\mathrm{rank}(\rN^1(F,\bZ))-1)$ on the N\'eron-Severi group.  

We use the partial description of extremal contractions \cite[1.1]{WW},
\cite[1.4,1.11]{Nami}, \cite{CMSB}.  The morphism $\beta':F \ra F'$
satisfies one of the following alternatives:
\begin{itemize}
\item{$\beta'$ is a divisorial contraction taking the
exceptional divisor to a surface $T\subset F'$.  At each smooth
point of $T$, $\beta'$ is locally a contraction to a two-dimensional
rational double point.}
\item{$\beta'$ is a small contraction, taking a smooth Lagrangian
$\bP^2 \subset F$ to an isolated singularity of $F'$.}
\end{itemize}

In the divisorial case, the smooth locus of $T$ has codimension $\ge 2$
complement and admits a holomorphic symplectic form.

Consider first the divisorial case.  
Suppose that $D$ is the exceptional divisor of $\beta$;
the generic fiber of $\beta|D:D\ra T$ is an ADE-configuration
of $\bP^1$'s.  Since $\beta$ is extremal, the fundamental
group of $T^{sm}$ acts transitively on the components
of $\beta^{-1}(t)$ for $t\in T$ generic.  An analysis
of intersection numbers implies
that only $A_1$ and $A_2$ configurations may occur (see \cite[5.1]{W}).

Let $\tilde{D}$ denote the normalization of $D$ and  
$$\tilde{D} \stackrel{\gamma}{\ra} \tilde{T} \ra T$$
the Stein factorization of $\beta|\tilde{D}$.  Then the generic fiber 
$C=\gamma^{-1}(t)$ is isomorphic to $\bP^1$.  However, the 
classification of rational double points yields
$$\cN_{\tilde{D}/F}|C \simeq \cO_{\bP^1}(-2),$$
hence
$$D.C =-1,-2.$$
This analysis does not require $F$ to be a numerical $\HKT$, only an 
irreducible holomorphic symplectic fourfold.  

We will now use integrality properties of the Beauville-Bogomolov form.
Let $\rho\in \rN^1(F,\bZ)$ denote the primitive class identified with 
a positive multiple of the extremal ray $R$ via the Beauville-Bogomolov form.  
Precisely, for each $A\in H^2(F,\bZ)$ we have
$$A.R=r\left(A,\rho\right)
$$
with $r=1,1/2$ depending on whether $\left(R,H_2(F,\bZ)\right)=\bZ,\frac{1}{2}\bZ.$
Clearly $C=mR$ and $D=n\rho$ for $m,n \in \bN$, and we have
$$D.C=mn R.\rho=mnr \left(\rho,\rho\right).$$
The following cases may occur:
\begin{enumerate}
\item[(I)]{$D.C=-1$:
\begin{enumerate}
\item{$r=1$:  Here $m=n=1$ and $R.\rho=-1$, hence $\left(\rho,\rho\right)=-1$
which is impossible because $\left(,\right)$ is even-valued.}
\item{$r=1/2$: Here $mn\left(\rho,\rho\right)=-2$ and thus 
$\left(\rho,\rho\right)=-2$. 
We conclude that $\left(R,R\right)=-1/2$.}
\end{enumerate}}
\item[(II)]{$D.C=-2$:
\begin{enumerate}
\item{$r=1$: Here we have $\left(\rho,\rho\right)=-2/mn$ which
forces $m=n=1$ and $\left(\rho,\rho\right)=-2$.  We conclude that
$\left(R,R\right)=-2$.}
\item{$r=1/2$:  Here $\left(\rho,\rho\right)=-4/mn$ so $mn=1$ or $2$.  
However, the lattice (\ref{eqn:thelattice})
does not admit primitive
vectors $\rho$ of length four with $\left(\rho,H^2(F,\bZ)\right)
=2\bZ$.  Indeed, if we had
$$\rho=2v+a\delta,  \quad  2\nmid a$$
with
$$v\in U^{\oplus 3} \oplus (-E_8)^{\oplus 2},
                \left(\delta,\delta\right)=-2, \left(v,\delta\right)=0,
$$
then it would follow that
\begin{equation} \label{eqn:congruence}
\left(\rho,\rho\right)=4\left(v,v\right)-2a^2\equiv -2\pmod{8}.
\end{equation}
We conclude that $mn=2$, 
$\left(\rho,\rho\right)=-2$, and $\left(R,R\right)=-1/2$.}
\end{enumerate}}
\end{enumerate}
This completes the proof in the divisorial case.  

We turn to the case where $\beta:F\ra F'$ is small 
contraction of a Lagrangian $\bP^2$.  Some multiple of the extremal ray $R$
is necessarily the class $L$ of a line in $\bP^2$.  
We shall show that $\left(L,L\right)=-5/2$ which implies
that $R=L$,  completing the proof of the theorem.

Suppose that $\lambda \in H^2(F,\bZ)$ is the unique class with
$$2A.L=\left(A,\lambda\right)
$$
for all $A\in H^2(F,\bZ)$.  We do not assume {\em a priori}
that $\lambda$ is primitive.
Consider a deformation $F_t$ of $F$ for which $[L]\in H_2(F_t,\bZ)$
(or equivalently, $\lambda$) remains a Hodge class.
The Lagrangian plane also deforms in $F_t$ (see \cite{voisin} and \cite{HT01}).  
For a general deformation $F_t$, the only Hodge classes in $H^4(F_t,\bZ)$ are
rational linear combinations  of $q^{\vee}$ and $\lambda^2$.
Indeed, the Torelli map is locally an isomorphism
and $q^{\vee},\lambda^2 \in H^4(F_t,\bQ)$
are the only Hodge classes in $\Sym^2 H^2(F_t,\bZ)$ for generic Hodge structures
on $H^2(F_t,\bZ)$ (see \cite[\S 3]{ogrady} for a detailed proof).  

We may put
\begin{equation} \label{eqn:form}
[\bP^2]=aq^{\vee}+b\lambda^2.
\end{equation}
Geometric properties of the Lagrangian plane translate into algebraic
conditions on the coefficients $a,b$;  we use
the intersection properties of numerical $\HKT$'s listed above:
\begin{itemize}
\item{The normal bundle to any Lagrangian submanifold is equal to its
cotangent bundle.  Thus we have 
$$\ [\bP^2].[\bP^2]=c_2(\Omega^1_{\bP^2})=3$$ 
which implies
$$25\cdot 23a^2+50ab\left(\lambda,\lambda\right)+3b^2\left(\lambda,\lambda\right)^2=3.$$
}
\item{Using the exact sequence
$$0 \ra \cT_{\bP^2} \ra \cT_F|\bP^2 \ra \cN_{\bP^2/F} \ra 0$$
we compute that $c_2(T_F)|\bP^2=-3$.  It follows that
$$-3=\frac{6}{5}(25\cdot 23 a  + 25b\left(\lambda,\lambda\right)).$$}
\item{We know that
$\lambda|\bP^2$ is some multiple of the hyperplane
class, i.e., $\lambda.[\bP^2] =\left( \begin{matrix} \lambda.L \end{matrix}\right) L$.  We deduce that
$$\lambda.\lambda.[\bP^2]=\left( \begin{matrix}\lambda.L \end{matrix} \right)^2 =\left(\lambda,\lambda\right)^2/4.$$
Using formula (\ref{eqn:form}) to evaluate $\lambda.\lambda.[\bP^2]$ we obtain
$$\left(\lambda,\lambda\right)^2/4=25a \left(\lambda,\lambda\right) + 
3b\left( \lambda , \lambda \right)^2.$$}
\end{itemize}
Altogether, we obtain three Diophantine equations in the variables
$\left(\lambda,\lambda\right),a,$ and $b$.  Eliminating $a$ and $b$ and
solving for $\left( \lambda,\lambda\right)$ we obtain the quadratic equation
$$23\left(\lambda,\lambda\right)^2 + 20 \left( \lambda,\lambda \right)- 2100=0$$
with solutions $\left(\lambda,\lambda\right)=-10,210/23$.  Only the first
makes sense.  We conclude that
$\left(L,L\right)=-5/2$ and $\lambda$ is primitive and 
$\left(\lambda,H^2(F,\bZ)\right)=2\bZ$. 
\end{proof}

\section{Applications to ample divisors}
\label{sect:half}

In this section, we prove one implication of Conjecture~\ref{conj:main}:
\begin{theo} \label{theo:main}
Let $(F,g)$ be a polarized irreducible holomorphic symplectic manifold,
deformation equivalent to Hilbert scheme of length-two subschemes of
a K3 surface.  Then we have
$$\NE_1(F)\subseteq \rN_E(F,g).$$
\end{theo}
Equivalently, the divisors predicted to be ample by our conjectures are indeed ample.

\begin{proof}
Let $M$ be a divisor such that
\begin{itemize}
\item{$\left(M,M\right)>0$; and}
\item{$M.R>0$ for each $R \in E^*$.}
\end{itemize}
The first condition implies that $M.R>0$ for $R \neq 0 \in \rN_1(F,\bZ)$ with
$\left(R,R\right) \ge 0$.  Indeed, $\left(,\right)$ has signature $(1,\dim \rN_1(F,\bR)-1)$
on $\rN_1(F,\bR)$.  
To prove the Theorem, it suffices to show that $M$ is ample on $F$.  

Suppose that $M$ fails to be ample.  After a small perturbation of $g$,
the line segment
$$tM+(1-t)g, \quad t \in [0,1]$$
meets the boundary of the ample cone of $F$ in the interior of the facet of the nef cone.
Indeed, Proposition~\ref{prop:chamber} shows that $\ocK_F \cap \cC_F$ 
is locally-finite rational polyhedral.  Under these assumptions
$$\tau:=\sup \{t: tM+(1-t)g \text{ is ample} \}$$
is rational.  Let $R$ be the (primitive, integral) generator of
the extremal ray corresponding to our facet;  we have $\left(R,R\right)<0$.  
Theorem~\ref{theo:classifyray} implies that
$$(R,R)=-1,-2,-5/2$$
whence $R \in \rN_E(F,g)$.  
\end{proof}

\begin{rema}
The underlying techniques here are reminiscent of those used in the proof that 
`minimal models are connected by flops' \cite{Kaw07} \cite[1.1.3]{BCHM}.  
\end{rema}

\bibliographystyle{plain}
\bibliography{DK3new}

\end{document}